\documentclass[12pt]{article}
\usepackage{graphicx}
\usepackage{tikz}
\usepackage{authblk,amsmath,amsthm,amsfonts,amssymb,amscd,enumerate,graphicx,color,appendix,tabularx,cite,booktabs}
\allowdisplaybreaks[4]
\newtheorem {Lemma}{Lemma}[section]
\newtheorem {Theorem} {Theorem}[section]
\newtheorem {Claim} {Claim}[section]

\newtheorem {Problem}{Problem}[section]
\numberwithin{equation}{section}
\usepackage{fullpage}

\allowdisplaybreaks [4]
\begin{document}

\title{Extremal distance spectral radius of graphs with fixed size}

\author{
Hongying Lin$^{a}$\footnote{E-mail: linhy99@scut.edu.cn},
Bo Zhou$^{b}$\footnote{
E-mail: zhoubo@m.scnu.edu.cn}  \\
$^{a}$School of Mathematics, South China  University of Technology, \\
Guangzhou 510641, P.R. China\\
$^{b}$School of  Mathematical Sciences, South China Normal University, \\
 Guangzhou 510631, P.R. China
 }

\date{}
\maketitle

\begin{abstract}
Let $m$ be a positive integer. Brualdi and Hoffman proposed the problem to determine the (connected) graphs with maximum spectral radius in a given graph class and they posed a conjecture for the class of graphs with given size $m$.
After partial results due to
Friedland and Stanley,
Rowlinson completely confirmed the conjecture.
The distance spectral radius of a connected graph is the largest eigenvalue of its distance matrix.
We investigate the problem to determine the connected graphs with minimum distance spectral radius in the class of graphs with size $m$.
Given $m$, there is exactly one  positive integer $n$ such that ${n-1\choose 2} <m\leq {n\choose 2}$.
We establish some structural properties of the extremal graphs for all $m$ and solve the problem for
${n-1\choose 2}+\max\{\frac{n-6}{2},1\}\le m\leq {n\choose 2}$. We give a conjecture for the remaining case.
To prove the main results, we also determine the the complements of  forests of fixed order with large and small distance spectral radius.
\\ \\
{\bf MSC:} 05C12, 05C50, 15A48\\ \\
{\bf Key words:} distance spectral radius, complement of a forest, graph size
\end{abstract}

\section{Introduction}

Throughout this paper, we consider simple graphs. Let $G$ be a connected graph with vertex
set $V(G)$ and edge set $E(G)$. Denote $m(G)$ the size of $G$.
The distance between vertices $u,v\in V(G)$,
denoted by $d_G(u,v)$, is the length of a shortest path between them.
The distance matrix of $G$, denoted by $D(G)$, is the matrix
$D(G)=(d_G(u,v))_{u,v\in V(G)}$.
Since $D(G)$ is real and
symmetric, its eigenvalues  are all real.
The  distance spectral radius of $G$, denoted by $\rho(G)$, is the largest
eigenvalue of $D(G)$. For $|V(G)|\ge 2$,  since $D(G)$ is
irreducible, we have by the Perron-Frobenius theorem that $\rho(G)$
is simple and positive, and there is a unique positive unit eigenvector $x(G)$ of $D(G)$
corresponding to $\rho(G)$, which we call the distance Perron
vector of $G$.

The study of  eigenvalues of the distance matrix of a connected graph  dates back to the
classical work of Graham and Pollak \cite{GP}, Graham and Lov\'asz~\cite{GL} and Edelberg et al.~\cite{EGG}.
The distance spectral radius was used as a
molecular descriptor \cite{BCM}.
Ruzieh and Powers \cite{SuPo90} and Stevanovi\' c and  Ili\' c \cite{SI} showed that among the connected
graphs of order $n$,
 the complete graph $K_n$ is the unique graph with
minimum distance spectral radius, while the path $P_n$ is the unique
graph with maximum distance spectral radius, where the uniqueness of the path was established in \cite{SI}. Related to many problems and properties of graphs such as Laplacian spectral properties \cite{Mer} and graph parameters based on distance \cite{AH0},
the distance eigenvalues and especially of the distance spectral radius
received much attention, see the surveys \cite{AH-2,LS} (and the references therein) and the very recent  \cite{CK}. We mention that such research was also extended to Laplacian and signless Laplacian versions \cite{RA},
hypergraphs \cite{WI} and other distances such as Steiner distance \cite{CT}.


Let $m$ be a positive integer. There is exactly one  positive integer $n$ such that ${n-1\choose 2} <m\leq {n\choose 2}$.
It is easily seen that  $\frac{1+\sqrt{8m}}{2}\leq n<\frac{3+\sqrt{8m+1}}{2}$.
Thus $n=\left\lceil\frac{1+\sqrt{8m}}{2}\right\rceil$.
After showing that the maximum adjacency spectral radius  of a graph with ${n\choose 2}$ edges, $n\ge 2$, is $n-1$, with equality precisely when the only nontrivial component of the graph is $K_{n}$,
Brualdi and Hoffman conjectured that among all graphs with $m$ edges, ${n-1\choose 2}<m<{n\choose 2}$, the maximal index is attained precisely when the only nontrivial component is the graph obtained from $K_{n-1}$ by adding one new vertex of degree $m-{n-1\choose 2}$. Friedland \cite{Fr} and Stanley \cite{St}
gave some partial results on this conjecture, which was confirmed finally  by
Rowlinson \cite{Row}.  Let $\mathbb{G}(m)$ be the set of connected graphs with $m$ edges.
Evidently, $ \mathbb{G}(m)=\{P_{m+1}\}$ when $m=1,2$. Thus, we assume in the following that $m\geq 3$. Then $m={n-1\choose 2}+s$ with $1\leq s\leq n-1$. So the graph $G_m$ obtained from $K_{n-1}$ by adding a new vertex of degree $s$ is the unique graph achieving the maximum adjacency spectral radius.

In this paper, we study the following problem.

\begin{Problem}
Determine the graphs in $\mathbb{G}(m)$ with maximum and minimum distance spectral radius.
\end{Problem}

It turns out that the maximum part is easy. However, the minimum is much more difficult. By existing results and experiments, the graphs with maximum (minimum, resp.) adjacency spectral radius happen to be the graphs with minimum (maximum, resp.) distance spectral radius over many classes of graphs, including trees and graphs with given parameters such as order, connectivity, clique number, etc. Our results show that this is not the case in the class $\mathbb{G}(m)$ for $s<n-2$. The main results are listed below.

For integer $n\geq 4$, let $A_n$ be the tree obtained from a path $v_1\ldots v_{n-1}$ by add a new vertex $v_n$ and a new edge $v_nv_2$.
For integer $n\geq 6$, let $B_n$ be the tree obtained from a path $v_1\ldots v_{n-1}$ by add a new vertex $v_n$ and a new edge $v_nv_3$.

\begin{Theorem}\label{max=1}
Let $G\in \mathbb{G}(m)\setminus \{P_{m+1}, A_{m+1}\}$, where $m\ge 5$.
Then
$\rho(G)\leq \rho(B_{m+1})<\rho(A_{m+1})<\rho(P_{m+1})$ with equality in the first inequality if and only if $G\cong B_{m+1}$.
\end{Theorem}

For a graph $G$, we denote by $\overline{G}$ the complement of $G$. Let $\Delta(G)$ and $\delta(G)$ be the maximum degree  and minimum degree of a graph $G$, respectively.

For positive integers $n\geq 2$ and $1\leq c\leq n$, write
\[
P_{n,c}=\overline{\left(c+c\lfloor\tfrac{n}{c}\rfloor-n\right)P_{\lfloor\frac{n}{c}\rfloor}\cup \left(n-c\lfloor\tfrac{n}{c}\rfloor\right)P_{\lceil\frac{n}{c}\rceil}}\,.
\]
Obviously, $P_{n,1}=\overline{P_n}$ and $P_{n,n}=K_n$.

\begin{Theorem}\label{size2023-result1=1}
Let $G\in \mathbb{G}(m)$, $n=\left\lceil\frac{1+\sqrt{8m}}{2}\right\rceil$ and
$s=m-{n-1\choose 2}$, where $m\ge 3$. Suppose that $\rho(G)$ is as small as possible. Then the order of $G$ is $n$,
and
\begin{enumerate}
\item[(i)] if $s\geq \frac{n-1}{2}$, then $\Delta(G)=n-1$;

\item[(ii)] if $s=\frac{n-2}{2}$ and $n$ is even, then $\Delta(G)=\delta(G)=n-2$;

\item[(iii)] if $1\leq s<\frac{n-2}{2} $, then
$\Delta(G)=n-2$ and $\delta(G)=n-3$,  i.e., each component of  $\overline{G}$
is either a cycle or a nontrivial paths, and there are exactly $s+1$ components being paths in $\overline{G}$.
\end{enumerate}
\end{Theorem}

\begin{Theorem}\label{size2025-result3}
Let $G\in \mathbb{G}(m)$, $n=\left\lceil\frac{1+\sqrt{8m}}{2}\right\rceil$ and
$s=m-{n-1\choose 2}$, where $m\ge 3$.
If $\max\{\frac{n-6}{2},1\}\leq s\leq n-1$, then $\rho(G)\geq \rho(P_{n,s+1})$ with equality if and only if $G\cong P_{n,s+1}$.
\end{Theorem}

To obtain these results, we also study the graphs which are the complements of  forests with large and small  distance spectral radius, respectively. At the end of the paper, we propose a conjecture on the extremal graphs for $s<\frac{n-6}{2}$.

\section{Preliminaries}

Let $G$ be a connected graph with $V(G)=\{v_1,\dots,v_n\}$. A column
vector $x=(x_{v_1},\dots, x_{v_n})^\top\in \mathbb{R}^n$ (whether it is the distance
Perron vector of $G$ or not) can be
considered as a function defined on $V(G)$ which maps vertex $v_i$
to $x_{v_i}$, i.e., $x(v_i)=x_{v_i}$ for $i=1,\dots,n$. Then
\[
x^\top D(G)x=\sum_{\{u,v\}\subseteq V(G)} 2d_G(u,v)x_ux_v.
\]
For a unit column vector $x\in\mathbb{R}^n$ with at least one nonnegative entry, by Rayleigh's principle, we
have
\[
\rho(G)\ge x^{\top}D(G)x
\]
with equality if and only if $x$ is the distance
Perron vector of $G$. In this case,
\[
\rho(G)
x_u=\sum_{v\in V(G)}d_G(u,v)x_v,
\]
which we call the $(\rho(G), x)$-eigenequation (at $u$).


For a connected graph  $G$ with $v\in V(G)$, let $\delta_G(v)$ be the  degree of  $v$ in $G$, $N_G(v)$ the set of neighbors of $v$ in $G$,
and $Tr_G(v)=\sum_{w\in V(G)}d_G(v,w)$, which is the transmission of $v$ in $G$. The Wiener index is the sum of the distances between all vertex pairs in $G$,
i.e., $W(G)=\frac{1}{2}\sum_{v\in V(G)}Tr_G(v)$. If the transmissions of all vertices in $G$ are equal, then we say $G$ is transmission regular.

Let $S_n$ and $C_n$ be the star and the cycle on $n$ vertices, respectively.

Let $G$ be a connected graph.
For $\emptyset\ne V_1\subset V(G)$, $G[V_1]$ denotes the subgraph of $G$ induced by $V_1$.
For $E_1\subseteq E(G)$, $G-E_1$ denotes the graph obtained from $G$ by deleting all edges of $E_1$. If $E_1=\{uv\}$, then we write $G-uv$ for $G-\{uv\}$. If $E'$ is a subset of edges of the complement of $G$, then $G+E'$ denotes the graph obtained from $G$ by inserting all edges of $E'$. If $E'=\{uv\}$, then we write $G+uv$ for $G+\{uv\}$.


The following lemma is an immediate consequence of the Perron-Frobenius theorem.

\begin{Lemma} \label{2.1} \cite{SuPo90}
Let G be a connected graph with $\{u,v\}\subset V(G)$. If $uv\not\in E(G)$,
then $\rho (G)>\rho (G+uv)$.
\end{Lemma}

\begin{Lemma} \label{low-bound} \cite{WZ}
Let $G$ be a connected graph of order $n$, and $Tr_{\min}(G)$ and $Tr_{\max}(G)$ be the minimum transmission of $G$ and the maximum transmission, respectively.
Then $Tr_{\min}(G)\leq \rho(G)\leq Tr_{\max}(G)$ and $\rho(G)\geq \frac{2W(G)}{n}$ with either equality   if and only if $G$ is transmission regular.
\end{Lemma}

\begin{Lemma} \label{permutation} \cite{XZ}
For a connected graph $G$ with $\eta$ being an
automorphism of $G$ and $x = x(G)$, $\eta(v_i) = v_j$ implies that $x_i = x_j$.
\end{Lemma}

\begin{Lemma}\label{trans1}
Let $G$ be a graph of diameter $2$ with  $x=x(G)$. Let $\{u,v\}\subset V(G)$
with  $\emptyset \ne S\subseteq N_G(u)\setminus (\{v\}\cup N_G(v))$.
Let $G'=G-\{uw: w\in S\}+\{vw: w\in S\}$.
If $x_u\geq x_v$ and the diameter of $G'$ is also two,
then $\rho(G')> \rho(G)$.
\end{Lemma}
\begin{proof}
Note that diameters of $G$ and $G'$ are two.
As we pass from $G$ to $G'$, the distance between a vertex of $S$ and $v$ is decreased by $1$,
the distance between a vertex of $S$ and $u$ is increased by $1$, and the distance between any other vertex pair remains unchanged.
Thus, we have by Rayleigh's principle that
\[\frac{1}{2}(\rho(G')-\rho(G))\geq \frac{1}{2}x^\top(D(G')-D(G))x=\sum_{w\in S}x_w(x_u-x_v)\geq 0,\]
so $\rho(G')\geq \rho(G)$. If $\rho(G')=\rho(G)$, then $\rho(G')=x^\top D(G')x$, so $x$ is also  the distance Perron vector of $G'$.
Then, from the $(\rho(G),x)$-eigenequation and the $(\rho(G'),x)$-eigenequation at $u$, we have
\[
0=\rho(G')x_u-\rho(G)x_u=\sum_{w\in S}x_w>0,
\]
a contradiction. It follows that $\rho(G')>\rho(G)$.
\end{proof}

A vertex of a graph is known as a pendant vertex or a leaf if it is degree is one. An edge incident to such a vertex is a pendant edge; otherwise, it is a non-pendant edge.

\section{Graphs which are the complements of  forests with large and small  distance spectral radius}

If $G$ is a graph and $uv$ is an edge of $G$ that is not a cut edge, then we denote by
 $G(uv)$ the graph obtained from $G$ by identifying the vertices $u$ and $v$ as vertex $u$, and adding a new pendent edge $uv$.

\begin{Lemma}\label{2024-CF-1}
Let $G$ be the complement of the union of $c$ disjoint graphs $G_1,\ldots, G_c$ of orders $n_1,\ldots,n_c$, respectively, where $c\geq2$.
If $G_1$ is a tree with a non-pendant edge $uv$,
then $\rho(G)< \rho\left(\overline{G_1(uv)\cup G_2\cup\cdots\cup G_c}\right)$.
\end{Lemma}
\begin{proof}
Let $G'=\overline{G_1(uv)\cup G_2\cup\cdots\cup G_c}$. Obviously, the diameter of $G'$ is two.
Let $x=x(G)$.
Suppose without loss of generality that $x_u\geq x_v$.
By Lemma~\ref{trans1}, $\rho(G')>\rho(G)$.
\end{proof}

A tree of order $n$ with diameter three is a double star $D_{n,a}$, obtainable by adding an edge between the centers of two disjoint stars $S_{a+1}$ and $S_{n-a-1}$ for some $a$ with $2a\le n-2$.

Lin and Drury \cite{ld}  proved among the graphs of order $n$ whose complements are trees different from the star, $D_{n,1}$ and $P_n$ are the graphs achieving maximum  and minimum distance  spectral radius, respectively. We establish results of the complements of forests.

\begin{Theorem}\label{forest-min-1}
Let $G$ be the complement of a forest on $n$ vertices with $c$ components, where $c\geq 2$.
If  $G\ncong \overline{S_{n-c+1}\cup (c-1)K_1}$, then
\[
\rho(G)\leq \rho(\overline{D_{n-c+1,1}\cup (c-1)K_1})<\rho(\overline{S_{n-c+1}\cup (c-1)K_1})
 \]
 with equality in the first inequality if and only if $\overline{G}\cong D_{n-c+1,1}\cup (c-1)K_1$.
\end{Theorem}
\begin{proof}
By Lemma~\ref{2024-CF-1},  $\rho(\overline{S_{n-c+1}\cup (c-1)K_1})>\rho(\overline{D_{n-c+1,1}\cup (c-1)K_1})$.

Let $G$ be the complement of a forest on $n$ vertices with $c$ components maximizing the distance spectral radius. Let $x=x(G)$.
Let $T_1,\ldots, T_c$ be the components of $\overline{G}$ with orders $n_1,\ldots n_c$, respectively,  where $n_1\geq \cdots\geq n_c\geq1$ and $n_1+\cdots+n_c=n$.

\noindent
{\bf Case~1.} $n_2\geq 2$.

\begin{Claim} \label{c1.1}
Each  $T_i$ for $1\leq i\leq c$ is a star.
\end{Claim}

\begin{proof}
Suppose that some $T_i$ is not a star.   By relabelling the components of $\overline{G}$, we may assume $i=1$.
Let $v_0\ldots v_d$ be a diametral path in $T_1$, where $d\geq 3$.
Let  $G'=\overline{T_1(v_2v_1)\cup T_2\cup \cdots\cup T_c}$\,.
Then, by Lemma~\ref{2024-CF-1}, $\rho(G')>\rho(G)$, a contradiction.
\end{proof}

\begin{Claim} \label{c1.2}
$n_3=\cdots=n_c=1$ and $n_2=2$.
\end{Claim}

\begin{proof}
Suppose that $n_3\geq 2$. By Claim~\ref{c1.1}, $T_{1}\cong S_{n_1}$ and $T_{3}\cong S_{n_3}$.
Let $u_0$ and $v_0$ be the centers of $S_{n_1}$ and $S_{n_3}$, respectively. Let $u_0u$ and $v_0v$ be fixed edges in $T_1$ and $T_3$, respectively.
By Lemma~\ref{permutation} and the $(\rho(G), x)$-eigenequation of $G$ at $u,v,u_0$ and $v_0$,
\begin{align}
\rho(G)x_u&=\sum_{w\in V(G)}x_w-x_u+x_{u_0}, \label{eq1}\\
\rho(G)x_{u_0}&=\sum_{w\in V(G)}x_w-x_{u_0}+(n_1-1)x_u, \label{eq2}\\
\rho(G)x_v&=\sum_{w\in V(G)}x_w-x_v+x_{v_0},\label{eq3}\\
\rho(G)x_{v_0}&=\sum_{w\in V(G)}x_w-x_{v_0}+(n_3-1)x_v. \label{eq4}
\end{align}
From \eqref{eq1} and \eqref{eq2}, and \eqref{eq3} and \eqref{eq4}, we have
\begin{align}
x_u&=\frac{\rho(G)+2}{\rho(G)+n_1}x_{u_0}, \label{eq5}\\
x_v&=\frac{\rho(G)+2}{\rho(G)+n_3}x_{v_0}.  \label{eq6}
\end{align}
From \eqref{eq2}, \eqref{eq4}, \eqref{eq5} and \eqref{eq6},
we have
\begin{align*}
\frac{x_{u_0}}{x_{v_0}}
&=1+\frac{(n_1-n_3)(\rho^2(G)+3\rho(G)+2)}{(\rho^2(G)+2\rho(G)-n_1+2)(\rho(G)+n_3)}
\geq1,
\end{align*}
so $x_{u_0}\geq x_{v_0}$.
Let $G'=G-u_0v+v_0v$. Note that both diameters of $G'$ and $G$ are two.
By Lemma~\ref{trans1}, we have  $\rho(G')>\rho(G)$, a contradiction.
Thus $n_3=1$.

By similar argument as above, we can get $n_2=2$.
\end{proof}

Combining Claims~\ref{c1.1} and~\ref{c1.2},  $G\cong \overline{S_{n-c}\cup K_2\cup (c-2)K_1}$.

\noindent{\bf Case~2.} $n_2=1$.

Note that $c\leq n-3$ as $G\ncong \overline{S_{n-c+1}\cup (c-1)K_1}$.

\begin{Claim}\label{c21}
$T_1$ is a double star.
\end{Claim}

\begin{proof}
Suppose that $T_1$ is not a double star. Then the diameter of $T_1$ is at least $4$.
Let $v_0v_1\ldots v_d$ be a diametral path in $T_1$, where $d\geq 4$.
Let $G'=\overline{T_1(v_2v_1)\cup T_2\cup \cdots\cup T_c}$.
Then, by Lemma~\ref{2024-CF-1}, we have $\rho(G')>\rho(G)$, a contradiction.
\end{proof}


By Claim~\ref{c21}, $T_1$ is a double star, say $T_1\cong D_{n-c+1,a}$, where
$1\leq a\leq \frac{n-c+1}{2}$.
Suppose that $a\geq 2$.
Let $v_1v_2v_3v_4$ be a diametral path of $T_1$ and let $x=x(G)$.
Let
\[
T'_1=T_1-\{wv_2:w\in N_{T_1}(v_2)\setminus\{v_1\}\}+\{wv_3:w\in N_{T_1}(v_2)\setminus\{v_1\}\}
 \]
 if $x_{v_2}\leq x_{v_3}$
and
\[
T'_1=T_1-\{wv_3:w\in N_{T_1}(v_3)\setminus\{v_4\}\}+\{wv_2:w\in N_{T_1}(v_3)\setminus\{v_4\}\}
 \]
if $x_{v_2}> x_{v_3}$.
Obviously, $T'_1\cong D_{n-c+1,1}$.
Let $G'=\overline{T'_1\cup (c-1)K_1}$.
By Lemma~\ref{trans1}, $\rho(G')>\rho(G)$, a contradiction. So $a=1$.
That is,  $G\cong \overline{D_{n-c+1,1}\cup (c-1)K_1}$.

Combining Cases 1 and 2, $G\cong H:=\overline{S_{n-c}\cup K_2\cup (c-2)K_1}$ or $H':=\overline{D_{n-c+1,1}\cup (c-1)K_1}$.
We need to compare $\rho(H)$ and $\rho(H')$.
Let $y=x(H)$.
Let $y_1$ and $y_2$ be the entries of $y$ corresponding to a leaf and the center of $S_{n-c}$, respectively,
$y_3$ be the entry of $y$ corresponding to a vertex of $K_2$, and $y_4$ be the entry of $y$ corresponding to a vertex of $(c-2)K_1$.
By Lemma~\ref{permutation} and the $(\rho(H), y)$-eigenequations, we have
\begin{align*}
\rho(H)y_1&=\sigma_y-y_1+y_2,\\
\rho(H)y_2&=\sigma_y-y_2+(n-c-1)y_1,\\
\rho(H)y_3&=\sigma_y-y_3+y_3,\\
\rho(H)y_4&=\sigma_y-y_4,
\end{align*}
where $\sigma_y=(n-c-1)y_1+y_2+2y_3+(c-2)y_4$.
Viewing the above four equations as a homogeneous linear system in the four variables $y_1, y_2,y_3$ and $y_4$, it has a nontrivial solution, so the determinant of the matrix of coefficients of
this homogeneous linear system is zero. A direct calculation show that the determinant is equal
to $P_H(t)$, where \begin{align*}
P_H(t)
=t^4-t^3(n-3)-t^2(5n-3c-4)+t(cn-4n-c^2+2c)+ 2n-2c-4.
\end{align*}
Then $\rho(H)$ is the largest root of $P_H(t)=0$.
By similar argument,
$\rho(H')$ is the largest root of $P_{H'}(t)=0$,
where
\begin{align}
P_{H'}(t)&=t^5-t^4(n-5)-t^3(7n-3c-10)+ t^2(cn-14n-c^2+8c+6)\notag\\
&\quad +t(2cn-8n- 2c^2+ 4c-9)+n-11 \notag\\
&=P_{H}(t)(t+2)-\left(2t^2+(2n-2c+5)t+3n-4c+3\right). \label{Eq2024-1227}
\end{align}
As $c\leq n-3$ and $\rho(H)>n$, we have
by \eqref{Eq2024-1227} that
\begin{align*}
P_{H'}(\rho(H))=&P_{H}(\rho(H))(\rho(H)+2)-\left(2\rho^2(H)+(2n-2c+5)\rho(H)+3n-4c+3\right)\\
=&-\left(2\rho^2(H)+(2n-2c+5)\rho(H)+3n-4c+3\right)\\
<&-\left(2n^2+(2n-2c+5)n+3n-4c+3\right)<0,
\end{align*}
which implies $\rho(H')>\rho(H)$.
Thus $G\cong \overline{D_{n-c+1,1}\cup (c-1)K_1}$, as desired.
\end{proof}

\begin{Theorem}\label{forest-min}
If $G$ is the complement of a forest on $n$ vertices with $c$ components, where $c\geq 2$,
then $\rho(G)\geq \rho(P_{n,c})$ with equality if and only if $G\cong P_{n,c}$.
\end{Theorem}
\begin{proof}
Let $G$ be the complement of a forest on $n$ vertices with $c$ components that minimizes the distance spectral radius.
Let $T_1,\ldots,T_c$ be all components of $\overline{G}$, and
$n_i$ be the order of $T_i$, where $i=1,\ldots,c$ and $n_1\geq \cdots\geq n_c\geq 1$.

\begin{Claim} \label{ca1}
Each $T_i$ for $i=1,\ldots,c$ is a path.
\end{Claim}

\begin{proof}
Suppose that some $T_i$, say $T_1$ is not a path.
Let $v$ be a pendant vertex in $T_1$ and $u$ be the vertex of degree at least $3$ such that $d_{T_1}(v,u)$ is as large as possible.
Then there are at least two pendant paths, say $P$ and $Q$ of the lengths of $p$ and $q$ respectively at $u$ in $T_1$, where $p\geq q\geq 1$.
Let $P=u_0u_1\ldots u_p$ and $Q=v_0v_1\ldots v_q$, where $u_0=v_0=u$.
Let $T'_1=T_1-v_qv_{q-1}+v_qu_p$ and $G'=\overline{T'_1\cup T_2\cup \cdots \cup T_c}$. Obviously, both diameters of $G'$ and $G$ are two.
Write $u_{p+1}=v_q$ in $G'$ and $x=x(G')$.

First, we prove that $x_{u_{p-i}}>x_{v_{q-1-i}}$ for all $i=0,1,\ldots, q-1$ by induction on $i$.

It is true for $i=0$, as otherwise, we have  $x_{u_{p}}\leq x_{v_{q-1}}$, so we have
by  Lemma~\ref{trans1} that $\rho(G)>\rho(G')$ for $G'=G-u_pu_{p+1}+v_{q-1}u_{p+1}$, contradicting the choice of $G$.

If $q=1$, then $i=0$.
Assume that $q\geq 2$ and $x_{u_{p-i}}>x_{v_{q-1-i}}$ for some $i=0,1,\ldots, q-2$.
Suppose that $x_{u_{p-(i+1)}}\leq x_{v_{q-1-(i+1)}}$.
Note that
\[
T_1\cong T'_1-u_{p-(i+1)}u_{p-i}- v_{q-1-(i+1)}v_{q-1-i}+u_{p-i}v_{q-1-(i+1)}+ u_{p-(i+1)}v_{q-1-i},
\]
and
\[
G\cong G'-u_{p-i}v_{q-1-(i+1)}-u_{p-(i+1)}v_{q-1-i}+u_{p-(i+1)}u_{p-i}+v_{q-1-(i+1)}v_{q-1-i}.
\]
Then
we have
\begin{align*}
\frac{1}{2}(\rho(G)-\rho(G'))&\geq \frac{1}{2}x^\top(D(G)-D(G'))x\\
&=\left(x_{v_{q-1-(i+1)}}-x_{u_{p-(i+1)}}\right)\left(x_{u_{p-i}}-x_{v_{q-1-i}}\right)\geq 0,
\end{align*}
i.e., $\rho(G)\geq \rho(G')$. Suppose that $\rho(G')= \rho(G)$. Since $\rho(G)=x^\top D(G)x$, $x$ is the distance Perron vector of $G$.
From the $(\rho(G),x)$-eigenequation and the $(\rho(G'),x)$-eigenequation at $u_{p-(i+1)}$,
we have
\begin{align*}
0=(\rho(G)-\rho(G'))x_{u_{p-(i+1)}}=x_{v_{q-1-i}}-x_{u_{p-i}}<0,
\end{align*}
a contradiction.
Thus $\rho(G)>\rho(G')$, a  contradiction again. So  $x_{u_{p-(i+1)}}>x_{v_{q-1-(i+1)}}$.
Therefore we indeed have $x_{u_{p-i}}>x_{v_{q-1-i}}$ for all $i=0,1,\ldots, q-1$.

Now, for $i=q-1$, we have $x_{u_{p-q+1}}>x_{v_0}=x_u$.
Note that
\[
T_1\cong T'_1-\{uw:w\in N_{T_1}(u)\setminus\{v_1,u_1\}\}+\{u_{p-q+1}w:w\in N_{T_1}(u)\setminus\{v_1,u_1\}\}
\]
and
\[
G\cong G'-\{u_{p-q+1}w:w\in N_{T_1}(u)\setminus\{v_1,u_1\}\}+\{uw:w\in N_{T_1}(u)\setminus\{v_1,u_1\}\}.
\]
By Lemma~\ref{trans1}, we have $\rho(G)>\rho(G')$, contradicting the choice of $G$.
\end{proof}

\begin{Claim}\label{ca2}
$n_1-n_c\leq 1$.
\end{Claim}

\begin{proof}
Suppose that $n_1-n_c\geq 2$.

Write $p=n_1$ and $q=n_c$.
By Claim~\ref{ca1}, we can assume that $T_1=u_1u_2\ldots u_p$ and $T_c=v_1v_2\ldots v_q$.
Let $G'=G-u_pv_q+u_pu_{p-1}$. Obviously, $\overline{G'}[V(T_1\cup T_c)]\cong P_{p-1}\cup P_{q+1}$ and the diameter of $G'$ is two.
Write $v_{q+1}=u_{p}$ in $G'$. Let $x=x(G')$.

If $x_{u_{p-1}}\geq x_{v_q}$, then as $G=G'-v_{q+1}u_{p-1}+v_{q+1}v_{q}$, we have
by Lemma~\ref{trans1} that $\rho(G)>\rho(G')$, a contradiction.
So  $x_{u_{p-1}}< x_{v_q}$.

First, we prove that $x_{u_{p-1-i}}<x_{u_{q-i}}$ for all $i=0,1,\ldots, q-1$ by induction on $i$.
If $q=1$, then $i=0$.
Assume that $q\geq 2$ and $x_{u_{p-1-i}}<x_{v_{q-i}}$ for all $i=0,1,\ldots, q-2$.
Suppose that  $x_{u_{p-1-(i+1)}}\geq x_{v_{q-(i+1)}}$.
Let
\[
G''=G'-u_{p-1-(i+1)}v_{q-i}-u_{p-1-i}v_{q-(i+1)}+u_{p-1-(i+1)}u_{p-1-i}+v_{q-(i+1)}v_{q-i}.
\]
Obviously, $\overline{G''}[V(T_1\cup T_c)]=P_p\cup P_q$ and $G''\cong G$.
Then
\begin{align*}
\frac{1}{2}(\rho(G)-\rho(G'))=\frac{1}{2}(\rho(G'')-\rho(G'))&\geq \frac{1}{2}x^\top(D(G'')-D(G'))x\\
&=\left(x_{u_{p-1-(i+1)}}-x_{v_{q-(i+1)}}\right)\left(x_{v_{p-i}}- x_{u_{p-1-i}}\right)\geq 0,
\end{align*}
i.e., $\rho(G)\geq \rho(G')$.
Suppose that $\rho(G)=\rho(G')$.
 Since $\rho(G)=x^\top D(G)x$, $x$ is the distance Perron vector of $G$.
From the $(\rho(G'),x)$-eigenequation and the $(\rho(G),x)$-eigenequation at $u_{p-1-(i+1)}$,
\begin{align*}
0=(\rho(G')-\rho(G))x_{u_{p-1-(i+1)}}&=x_{v_{q-i}}-x_{u_{p-1-i}}>0,
\end{align*}
a contradiction. So $\rho(G)>\rho(G')$, a contradiction again.
Thus $x_{u_{p-1-(i+1)}}< x_{v_{q-(i+1)}}$.
Therefore we indeed have  $x_{u_{p-1-i}}<x_{v_{q-i}}$ for all $i=0,1,\ldots, q-1$.

Now, for $i=q-1$, we have $x_{u_{p-q}}<x_{v_{1}}$.
Let $G''=G'-u_{p-q-1}v_1 +u_{p-q-1}u_{p-q}$.
Evidently, $\overline{G''}[V(T_1\cup T_c)]=P_p\cup P_q$, and $G\cong G''$.
By Lemma~\ref{trans1}, we have $\rho(G)=\rho(G'')>\rho(G')$, also a contradiction.
\end{proof}

Combining Claims~\ref{ca1} and~\ref{ca2}, we have $G\cong P_{n,c}$.
\end{proof}

\section{Proof of Theorem \ref{max=1}}

\begin{Lemma}\label{p}\cite{SI, WZ}
Let $G$ be a tree on $n$ vertices, where $n\geq 6$.
If $G\notin\{P_n,A_n\}$, then $\rho(G)\leq \rho(B_n)< \rho(A_n)< \rho(P_n)$ with equality if and only if $G\cong B_n$.
\end{Lemma}


\begin{proof}[Proof of Theorem \ref{max=1}]
Let $G$ be the graph different from $P_{m+1}, A_{m+1}$ in $\mathbb{G}(m)$ maximizing the distance spectral radius. Let $n$ be the order of $G$.
As $D(P_m)$ is a submatrix of each of $D(A_{m+1}), D(B_{m+1})$ and $D(P_{m+1})\}$,
we have by Lemma~\ref{p} that
\[
\rho(P_m)<\min\{\rho(A_{m+1}),\rho(B_{m+1}),\rho(P_{m+1})\}=\rho(B_{m+1}).
\]
If $n\leq m$, then we have by Lemma~\ref{p} that $\rho(G)\leq \rho(P_m)<\rho(B_{m+1})$, a contradiction.
Thus $n=m+1$. By Lemma~\ref{p}, $\rho(G)\leq \rho(B_{m+1})$ with equality if and only if $G\cong B_{m+1}$.
\end{proof}

\section{Proof of Theorems \ref{size2023-result1=1}}


\begin{proof}[Proof of Theorems \ref{size2023-result1=1}]
Let $n_G$ be the order of $G$.

\begin{Claim}\label{order} $n_G=n$.
\end{Claim}
\begin{proof}
Since   ${n-1\choose 2} <m\leq {n\choose 2}$,
we have $n_G\geq n$.
As $\rho(G)$ is as small as possible and $P_{n,s+1}\in \mathbb{G}(m)$, we have $\rho(G)\leq \rho(P_{n,s+1})$.
Note that by Lemma~~\ref{low-bound},
\[\rho(P_{n,s+1})\left\{\begin{array}{ll}
                                        =Tr_{\max}(P_{n,s+1})=n-1& \hbox{if $s=n-1$,}
                                       \\<Tr_{\max}(P_{n,s+1})=n & \hbox{if $\frac{n-1}{2}\leq s\leq n-2$,} \\
                                        =Tr_{\max}(P_{n,s+1})=n & \hbox{if $s=\frac{n-2}{2}$ and $n$ is even,}\\
                                        <Tr_{\max}(P_{n,s+1})=n+1  & \hbox{if $1\leq s\leq \frac{n-3}{2}$.}
                                      \end{array}
                                    \right.
\]
If $s=n-1$, then $G\cong K_n=P_{n,n}$,
as otherwise, $n_G\geq n+1$, so we have by Lemma~\ref{low-bound} that
\[\rho(G)\geq Tr_{\min}(G)\geq n_G-1\geq n>n-1=\rho(P_{n,n}),\]
 contradicting the choice of $G$.
Assume  $1 \leq s\leq n-2$.
If $n_G\geq n+1$,
then by Lemma~\ref{low-bound},
\begin{align*}
\rho(G)\geq \frac{2W(G)}{n_G}&\geq  \frac{2\left(m+2\left({n_G\choose 2}-m\right)\right)}{n_G}\\
&=\frac{4{n_G\choose 2}-2m}{n_G}\\
&=2(n_G-1)-\frac{2m}{n_G}\\
& \ge  2n-\frac{n(n-1)}{n+1}\\
& > n+1\\
& >\rho(P_{n,s+1}),
\end{align*}
a contradiction.
Therefore $n_G=n$.
\end{proof}

Suppose first that $s\geq  \frac{n-1}{2} $.
If $\Delta(G)\leq n-2$, then we have by Lemma~\ref{low-bound} that
\[\rho(G)\geq Tr_{\min}(G)\geq \Delta(G)+2(n-1-\Delta(G))\geq n>\rho(P_{n,s+1}), \]
a contradiction. So $\Delta(G)=n-1$. This proves (i).

Suppose next that $s=\frac{n-2}{2}$ and $n$ is even.
If $\Delta(G)\leq n-3$, then we have by Lemma~\ref{low-bound} that
\[\rho(G)\geq Tr_{\min}(G)\geq \Delta(G)+2(n-1-\Delta(G))\geq n+1>\rho(P_{n,s+1}), \]
a contradiction.
Suppose that $\Delta(G)= n-1$.
Then $G\cong K_1\vee H$, where $H$ is a graph of order $n-1$ with $m-n+1$ edges.
Since $m(G)=m(P_{n,\frac{n}{2}})$ and each vertex of $P_{n,\frac{n}{2}}$ is of degree $n-2$,
the minimum degree of $G$ is at most $n-3$.
Let $x=x(P_{n,\frac{n}{2}})$. By Lemma~\ref{permutation}, $x=\frac{1}{\sqrt{n}}(1,1,\ldots,1)^\top$.
Then we have
\begin{align*}
\rho(G)\geq x^\top D(G)x=\rho(P_{n,\frac{n}{2}}).
\end{align*}
Suppose that $\rho(G)=\rho(P_{n,\frac{n}{2}})$.
Then $x=x(G)$.
By the $(\rho(G),x)$-eigenequation at a vertex with minimum degree, we have
\[
{\rho(G)}={2n-2-\delta(G)}
\geq {2n-2-(n-3)}
={n+1}
>{\rho(P_{n,\frac{n}{2}})},
\]
also a contradicton.
It follows that $\Delta(G)=n-2$. Then $2m=n(n-2)=n\Delta$, so $\Delta(G)=\delta(G)=n-2$. This is (ii).

Now suppose that $1\leq s\leq  \frac{n-3}{2}$.
Let $v_1, \ldots, v_n$ be the  vertices of $G$,  where $\delta_G(v_1)\geq \cdots \geq \delta_G(v_n)$.

\begin{Claim} \label{njian2}  $\Delta(G)=n-2$.
\end{Claim}

\begin{proof}
If $\Delta(G)\leq n-3$, then \[m\leq \frac{n(n-3)}{2}<{n-1\choose 2},\]
a contradiction. Thus $\Delta(G)=n-2,n-1$.

Suppose that $\Delta(G)=n-1$. Let $a_1$ be the number of vertices with degree $n-1$ in $G$.
Then $G\cong K_{a_1}\vee H$ for some graph $H$ of order $n-a_1$  with size $m-{a_1 \choose 2}-a_1(n-a_1)$.

Suppose that $\Delta(\overline{H})\geq 3$.
Let $u\in V(H)$ with $\delta_{\overline{H}}(u)=\Delta(\overline{H})$, and $u_1$ be a neighbor $u$ in $\overline{H}$ such that $\delta_{\overline{H}}(u_1)$ is as small as possible.
Let $G'=G-u_1v_1+uu_1$, $x=x(G')$ and $\sigma=\sum_{w\in V(G')}x_w$.
Let $u_2$ and $u_3$ be two neighbors  of $u$ in $\overline{G'}$. 
F$i=1,2,3$, we have from the $(\rho(G'),x)$-eigenequation at $u_i$ that
\begin{align*}
\rho(G')x_{u_i}&=\sigma-x_{u_i}+\sum_{w\in N_{\overline{G'}}(u_i)}x_w, 
\end{align*}
so
\begin{align*}
(\rho(G')+1)(x_{u_2}+x_{u_3}-x_{u_1})
&=\sum_{w\in N_{\overline{G'}}(u_2)}x_w+\sum_{w\in N_{\overline{G'}}(u_3)}x_w+\sigma-\sum_{w\in N_{\overline{G'}}(u_1)}x_w\\ &\geq  2x_{u}>0.
\end{align*}
Then $\sum_{w\in N_{\overline{G'}}(u)}x_w-x_{u_1}\geq x_{u_2}+x_{u_3}-x_{u_1}>0$.
By the $(\rho(G'),x)$-eigenequation at $u$ and $v_1$,
\begin{align*}
\rho(G')x_{u}&=\sigma-x_u+\sum_{w\in N_{\overline{G'}}(u)}x_w,\\
\rho(G')x_{v_1}&=\sigma-x_{v_1}+x_{u_1}.
\end{align*}
So  $(\rho(G')+1)(x_{u}-x_{v_1})=\sum_{w\in N_{\overline{G'}}(u)}x_w-x_{u_1}>0$, i.e., $x_{u}>x_{v_1}$.
By Lemma~\ref{trans1}, we have $\rho(G)>\rho(G')$, a contradiction.
Thus $\Delta(\overline{H})\leq 2$.
It follows that each component of $\overline{H}$ is either a cycle or a path.

Suppose there is a cycle component, say $H_1:=u_1u_2\ldots u_{n_1}u_1$ in $\overline{H}$,
where $n_1$ is the order of $H_1$ and $n_1\geq 3$.
Let $G''=G-u_1v_1+u_1u_2$ and $y=x(G'')$.
By Lemma~\ref{permutation}, $y_{v_1}=y_{u_2}$.
By Lemma~\ref{trans1}, we have $\rho(G)>\rho(G'')$, a contradiction.
So each component of $\overline{H}$ is a path.
Thus each component of $\overline{G}$ is a path.
It follows that $\overline{G}$ is a forest with $s+1$ components. Note that $G\ncong P_{n,s+1}$ since $\Delta(P_{n,s+1})=n-2$.
By Theorem~\ref{forest-min}, $\rho(G)>\rho(P_{n,s+1})$, a contradiction.
Therefore $\Delta(G)=n-2$.
\end{proof}

By Claim~\ref{njian2}, each vertex is  incident to at least one edge in  $\overline{G}$, so each component in $\overline{G}$ is nontrivial.
Let $t$ be the number of components in $\overline{G}$.
Let $H_i$ with $i=1,\ldots, t$ be the component of order $n_i$ in $\overline{G}$,
where $\sum_{i=1}^tn_i=n$ and $n_i\geq 2$ for $i=1,\ldots,t$.

If some component $\overline{G}$ contains a cycle, assuming that $H_i$ contain a cycle in $\overline{G}$ for  $i=1,\ldots,c$, where $c$ is the number of components containing a cycle in $\overline{G}$, then
\begin{align*}
 n-1-s=\sum_{i=1}^t|E(H_i)|=&\sum_{i=1}^c|E(H_i)|+\sum_{i=c+1}^t|E(H_i)|\\
\geq &\sum_{i=1}^cn_i+\sum_{i=c+1}^t(n_i-1)=n-(t-c),
\end{align*}
so we have $t-c\geq s+1$, i.e., $c\le t-s-1$ with equality if and only if  $H_i$ for each $i=1,\ldots,c$ is a unicyclic graph.

By similar argument as  Claim~\ref{ca1} in the proof of Theorem~\ref{forest-min}, we have

\begin{Claim} \label{njian4} $H_i$  is a path for $i=c+1,\ldots,t$.
\end{Claim}

\begin{Claim} \label{njian5} $\delta(G)=n-3$.
\end{Claim}
\begin{proof}
If $c=0$, then $\overline{G}$ is a forest with $s+1$ components,  so we have by Theorem~\ref{forest-min} that  $G\cong P_{n,s+1}$.
It follows that $\delta(G)=n-3$, as desired.

Assume that $c\geq 1$.
Suppose that $\delta(G)\leq n-4$.
As $c\le t-s-1$, we have
\begin{align*}
\sum_{i=1}^{t-s-1}|E(H_i)|&=n-1-s-\sum_{i=t-s}^t|E(H_i)|\\
&= n-1-s-\sum_{i=t-s}^t(n_i-1)\\
&=\sum_{i=1}^{t-s-1}n_i.
\end{align*}
Let $b=\sum_{i=1}^{t-s-1}n_i$.
Let $G'=\overline{C_b\cup H_{t-s}\cup \cdots\cup H_t}$ and $x=x(G')$. By Lemma~\ref{permutation},
the entry of each vertex on the cycle in $\overline{G'}$ is the same, which we denote by $x_0$.
As we pass from $G$ to $G'$, we have \[\rho(G)-\rho(G')\geq x^\top(D(G)-D(G'))x=0,\]
i.e., $\rho(G)\geq \rho(G')$.
Suppose that $\rho(G)=\rho(G')$.
As $\delta(G)\leq n-4$, we have by Claim~\ref{njian4} that,  for some $i=1,\ldots, c$,
$H_i$ contains a vertex with minimum degree in $G$.
By the $(\rho(G'),x)$-eigenequation and the $(\rho(G),x)$-eigenequation at a vertex with minimum degree in $G$, we have
\begin{align*}
0=\rho(G)x_0-\rho(G')x_0&=(n-3-\delta(G))x_0>x_0>0,
\end{align*}
a contradiction.
So  $\rho(G)>\rho(G')$, a contradiction again.  Thus $\delta(G)=n-3$.
\end{proof}

Combining Claims~\ref{njian2} and \ref{njian5},  $\delta(\overline{G})=1 $  and $\Delta(\overline{G})=2$,
so $H_i$  is a cycle for  each $i=1,\ldots,c$ and $t=c+s+1$.
By  Claim \ref{njian4}, $H_i$ is a nontrivial path for   each  $i=c+1,\ldots,c+s+1$.
We complete the proof of (iii).
\end{proof}

\section{Proofs of Theorem \ref{size2025-result3}}

%
%

For positive integer $a$, let $\widetilde{K}_{2a}$ be the graph obtained from $K_{2a}$ by deleting a perfect matching.
%

\begin{proof}[Proof of Theorems \ref{size2025-result3}]

Let $G$ be the graph that minimizes the distance spectral radius in $\mathbb{G}(m)$.
From Theorem~\ref{size2023-result1=1}, the order of $G$ is $n$.

\noindent{ \bf{Case~1.}} $\frac{n-1}{2}\leq s\leq n-1$.

It is trivial that  $G\cong K_n=P_{n,n}$ if $s=n-1$.

Suppose that $\frac{n-1}{2}\leq s\leq n-2$.
By Theorem~\ref{size2023-result1=1} (i), $\Delta(G)=n-1$.
Let $a$ be the number of vertices of degree $n-1$ in $G$.
If $a\leq 2s+1-n$, then \[m\leq \frac{(n-1)(2s+1-n)+(n-2)(2n-2s-1)}{2}=\frac{n^2-3n+1}{2}+s=m-\frac{1}{2},\]
a contradiction.
So $a\geq 2s+2-n$.
It follows that  $G\cong K_{2s+2-n}\vee H$, where $H$ is a  graph of order  $2n-2s-2$ with size $m-{2s+2-n\choose 2}-(2s+2-n)(2n-2s-2)$.

Suppose that $H\ncong \widetilde{K}_{2n-2s-2}$. Note that $P_{n,s+1}=K_{2s+2-n}\vee \widetilde{K}_{2n-2s-2}$.
Let $x=x(P_{n,s+1})$.
Let $x_1$ be the entry of $x$ corresponding to any vertex of degree $n-1$ in $P_{n,s+1}$,
and $x_2$ be the entry of $x$ corresponding to any vertex of degree $n-2$ in $P_{n,s+1}$.
Construct an unit vector $y$ of order $n$, where the entry of $y$ corresponding to any vertex of  $V(K_{2s+2-n})$ in $G$ is $x_1$, and the entry of $y$ corresponding to any vertex of $V(H)$ in $G$ is $x_2$. So
\begin{align*}
\rho(G)&\geq y^\top D(K_{2s+2-n}\vee H)y
=x^\top D(P_{n,s+1})x
=\rho(P_{n,s+1}),
\end{align*}
i.e., $\rho(G)\geq \rho(P_{n,s+1})$.
Suppose that $\rho(G)=\rho(P_{n,s+1})$. Then $y$ is also the distance Perron vector of $G$.
Label the vertices of $G$ as $v_1, \ldots, v_n$,  where $\delta_G(v_1)\geq \cdots \geq \delta_G(v_n)$.
Since $m(H)=m( \widetilde{K}_{2n-2s-2})$, $H\ncong \widetilde{K}_{2n-2s-2}$ and
 each vertex of $\widetilde{K}_{2n-2s-2}$ is of degree $2n-2s-4$,
 the minimum degree of $H$ is less than $2n-2s-4$.
By the $(\rho(G),y)$-eigenequation at $v_n$, we have
\begin{align*}
\rho(G)x_2&=\rho(G)y_{v_n}\\
&= (2s+2-n)x_1+(2(2n-2s-3)-\delta_H(v_n))x_2\\
&\geq (2s+2-n)x_1+(2(2n-2s-3)-(2n-2s-5))x_2\\
&=(2s+2-n)x_1+(2n-2s-1)x_2\\
& > (2s+2-n)x_1+(2n-2s-2)x_2\\
&=\rho(P_{n,s+1})x_2,
\end{align*}
a contradiction.
Thus  $\rho(G)>\rho(P_{n,s+1})$, which is a contradiction.
So  $H\cong  \widetilde{K}_{2n-2s-2}$.
Therefore $G\cong P_{n,s+1}$, as desired.

\noindent{ \bf{Case~2.}} $\frac{n-5}{2}\leq s\leq\frac{n-2}{2}$.

If $s=\frac{n-2}{2}$, then by Theorem~\ref{size2023-result1=1} (ii), $\Delta(G)=\delta(G)=n-2$,  so $G\cong P_{n,s+1}$.

Suppose that $\frac{n-6}{2}\leq s\leq \frac{n-2}{2}$. By Theorem \ref{size2023-result1=1} (iii), $\Delta(G)=n-2$ and $\delta(G)=n-3$.
Let $a_i$ be the number of vertices with degree $n-i$ in $G$, where $i=2,3$.
Then $a_2+a_3=n$ and $a_2+2a_3=2(n-s-1)$, so $a_2=2s+2$ and $a_3=n-(2s+2)$.
Note that if $\overline{G}$ is a forest, then  $G\cong P_{n,s+1}$ by Theorem~\ref{forest-min}.

If $s=\frac{n-3}{2}$, then $n$ is odd and $n\geq 5$, and as
$a_2=n-1$ and $a_3=1$, we have $G\cong P_{n,\frac{n-1}{2}}$.

If $s=\frac{n-4}{2}$, then $n$ is even and $n\geq 6$, and as
 $a_2=n-2$ and $a_3=2$, we have $G\cong P_{n,\frac{n-2}{2}}$.

Suppose that $s=\frac{n-5}{2}$. Then $n$ is odd and $n\geq 7$.
Since $a_2=n-3$ and $a_3=3$, we get
$G\cong P_{n,\frac{n-3}{2}}$ or $\overline{C_3\cup \frac{n-3}{2}K_2}$.
We need to compare $\rho\left(P_{n,\frac{n-3}{2}}\right)$ and $\rho\left(\overline{C_3\cup \frac{n-3}{2}K_2}\right)$.
If $n=7$, then by a direct calculation, $\rho\left(\overline{C_3\cup \frac{n-3}{2}K_2}\right)\approx 7.4641> 7.4553\approx \rho\left(P_{n,\frac{n-3}{2}}\right)$.
Suppose that $n\geq 9$. Then there are three paths of length $2$, say $v_{i1}v_{i2}v_{i3}$ with $i=1,2,3$ in $\overline{P_{n,\frac{n-3}{2}}}$.
Let $x=x\left(P_{n,\frac{n-3}{2}}\right)$. By Lemma~\ref{permutation}, $x_{v_{i1}}=x_{v_{i3}}$, $x_{v_{i1}}=x_{v_{j1}}$ and $x_{v_{ii}}=x_{v_{jj}}$,
where $i,j=1,2,3$.
Let
\begin{align*}
G'& =P_{n,\frac{n-3}{2}}-\{v_{i2}v_{j2}:1\leq i<j\leq 3\}-\{v_{i1}v_{i3}:i=1,2,3\}\\
&\quad +\{v_{i1}v_{i2}:i=1,2,3\}+\{v_{i2}v_{i3}:i=1,2,3\}.
\end{align*}
Evidently, $G'\cong \overline{C_3\cup \frac{n-3}{2}K_2}$.
Then
\begin{align*}
\frac{1}{2}\left(\rho(G')-\rho\left(P_{n,\frac{n-3}{2}}\right)\right)&\geq \frac{1}{2}x^\top\left(D(G')-D\left(P_{n,\frac{n-3}{2}}\right)\right)x\\
&=3x_{v_{12}}^2+3x_{v_{11}}^2-3x_{v_{11}}x_{v_{12}}-3x_{v_{11}}x_{v_{12}}\\
&=3(x_{v_{11}}-x_{v_{12}})^2\\
&\geq 0,
\end{align*}
so $\rho(G')\geq \rho\left(P_{n,\frac{n-3}{2}}\right)$.
If $\rho(G')= \rho\left(P_{n,\frac{n-3}{2}}\right)$, then $x_{v_{11}}=x_{v_{12}}$.
From the  $\left(\rho\left(P_{n,\frac{n-3}{2}}\right), x\right)$-eigenequation at $v_{11}$ and $v_{12}$,
we have
\[0=\left(\rho\left(P_{n,\frac{n-3}{2}}\right)+1\right)(x_{v_{11}}-x_{v_{12}})=-x_{v_{13}}<0,\]
a contradiction.
Thus $\rho\left(\overline{C_3\cup \frac{n-3}{2}K_2}\right)=\rho(G')>\rho\left(P_{n,\frac{n-3}{2}}\right)$.
Therefore $G\cong P_{n,\frac{n-3}{2}}$.

Finally, suppose that $s=\frac{n-6}{2}$. Then $n$ is even and $n\geq 8$.
Since $a_2=n-4$ and $a_3=4$, we get
$G\cong P_{n,\frac{n-4}{2}}$, $\overline{C_4\cup \frac{n-4}{2}K_2}$ or $\overline{C_3\cup P_3\cup \frac{n-6}{2}K_2}$. Correspondingly,  we have by a direct calculation that $\rho(G)=8.5249$,  $8.5311$,  $8.5283$ if $n=8$,
 $\rho(G)=10.4195$, $10.4244$, $10.4228$ if $n=10$. In either case,  $P_{n,\frac{n-4}{2}}$ uniquely minimizes the distance spectral radius.
%
%
%
%
%
%
%
%
If $n\geq 12$, then by similar argument as above, we have $\rho\left(\overline{C_4\cup \frac{n-4}{2}K_2}\right)>\rho\left(P_{n,\frac{n-4}{2}}\right)$ and $\rho\left(\overline{C_3\cup P_3\cup \frac{n-6}{2}K_2}\right)>\rho\left(P_{n,\frac{n-4}{2}}\right)$.
Thus  $G\cong P_{n,\frac{n-4}{2}}$.
\end{proof}

From the proof of Theorem~\ref{size2025-result3}, we have:
For integers $n$ and $m$ with ${n-1\choose 2}+\max\{\left \lfloor\frac{n-5}{2} \right\rfloor,1\}\leq m\leq {n\choose 2}$,
let $G$ be a graph on $n$ vertices with size $m$.
Then $\rho(G)\geq \rho(P_{n,s+1})$ with equality if and only if $G\cong P_{n,s+1}$, where $s=m-{n-1 \choose 2}$.

\section{Concluding remarks}

By Theorem \ref{size2025-result3}, the graph  $P_{n,s+1}$ uniquely minimizes the distance spectral radius
among $\mathbb{G}(m)$, where $n=\left\lceil\frac{1+\sqrt{8m}}{2}\right\rceil$,
$s=m-{n-1\choose 2}$, and $\max\{\frac{n-6}{2},1\}\leq s\leq n-1$.
The two smallest cases not covered by Theorem \ref{size2025-result3} is  $(m,n,s)=(29,9,1)$ and
$(m,n,s)=(37,10,1)$.
Let $G$ be a graph satisfying the conditions
in Theorem~\ref{size2023-result1=1}. For the former (latter, respectively) case, the possible graph $G$ and  the the value of $\rho(G)$ are listed in Table 1 (Table 2, respectively).
From the two tables, we find that $P_{n,s+1}$ uniquely minimizes the distance spectral radius.
Based on these observations further computer searching results,
So we pose the following conjecture:

Let $G\in \mathbb{G}(m)$, $n=\left\lceil\frac{1+\sqrt{8m}}{2}\right\rceil$ and $s=m-{n-1\choose 2}$, where $m\ge 3$ and $1\leq s\leq \frac{n-6}{2}$. Then  $\rho(G)\geq \rho(P_{n,s+1})$ with equality if and only if $G\cong P_{n,s+1}$.
If it is true, then by Theorem~\ref{forest-min}, to prove the conjecture, it suffices to show that
 each component of $\overline{G}$ is a tree.

\begin{table}[htbp]
\caption{$G$ and $\rho(G)$ when $n=9$ and $s=1$.}\label{T1}
\centering
\smallskip
\renewcommand\arraystretch{1.5}
\begin{tabular}{cccccc}
\toprule[1.5pt]
$G$ & $P_{9,2}$ & $\overline{C_5\cup2K_2}$ & $\overline{C_4\cup P_3\cup K_2}$ &
$\overline{C_3\cup 2P_3}$ & $\overline{C_3\cup P_4\cup P_2}$\\
$\rho (G)$ & $9.5782$ & $9.5826$ & $9.5806$ & $9.5786$ & $9.5804$\\
\bottomrule[1.5pt]
\end{tabular}
\end{table}
\begin{table}[htbp]
\caption{$G$ and $\rho(G)$ when $n=10$ and $s=1$.}\label{T2}
\centering
\smallskip
\renewcommand\arraystretch{1.5}
\begin{tabular}{cccccc}
\toprule[1.5pt]
$G$ & $P_{10,2}$ & $\overline{C_6\cup 2K_2}$ &  $\overline{C_5\cup P_3\cup K_2}$ & $\overline{C_4\cup P_4\cup P_2}$ \\
$\rho(G)$ & $10.6203$ & $10.6235$ & $10.6220$ & $10.6219$\\
\midrule[1.5pt]
$G$ & $\overline{C_4\cup 2P_3}$ & $\overline{2C_3\cup 2K_2}$ & $\overline{C_3\cup P_4\cup P_3}$ & $\overline{C_3\cup P_5\cup P_2}$\\
$\rho(G)$ & $10.6205$ &  $10.6235$ & $10.6204$ & $10.6219$\\
\bottomrule[1.5pt]
\end{tabular}
\end{table}

%

\vspace{5mm}

\noindent {\bf Acknowledgements.}
This work was supported by the National Natural Science Foundation of China (Nos.~12071158 and~11801410),
the Guangdong Basic and Applied Basic Research Foundation (No.~2024A1515010493)
and the Guangzhou Basic and Applied Basic Research Foundation (No.~2024A04J3485). The second author once discussed the problem with Professor Linyuan Lu.

\end{document}